\documentclass[12pt]{amsart}
\usepackage[english]{babel}
\usepackage{amssymb,amsmath,amsthm,enumerate}
\usepackage[latin2]{inputenc}
\usepackage{amsfonts, verbatim,  amscd}

\newtheorem{theorem}{Theorem}
\newtheorem{proposition}[theorem]{Proposition}

\newtheorem{coro}[theorem]{Corollary}
\def\acknowledgment{\par\addvspace{17pt}\small\rmfamily
\trivlist\if!\ackname!\item[]\else
\item[\hskip\labelsep
{\bfseries\ackname}]\fi}

\def\C{\mathbb{C}}
\def\R{\mathbb{R}}

\def\N{\mathbb{N}}

\newcommand{\du}{\mathrm{d}}

\begin{document}
\title[Local approximation of non-holomorphic discs in almost complex manifolds]{Local approximation of non-holomorphic discs in almost complex manifolds} 
\author{Florian Bertrand and Uro\v s Kuzman }

\begin{abstract}
We provide a local approximation result of 
non-holomorphic discs with small $\bar{\partial}$ by pseudoholomorphic ones. As an application, we provide a certain gluing construction.
\end{abstract} 


\maketitle

\section*{Introduction}
In \cite{Rosay2}, J.-P. Rosay stated the following problem for complex manifolds: can a smooth non-holomorphic disc $\varphi$ with a small  $\overline{\partial}\varphi$ always be approximated by a holomorphic one? The question is very general and, in fact, his paper itself contains a counterexample in a compact Riemann surface of genus $\geq 2$ (due to L. Lempert). However, under certain restrictions on the initial disc $\varphi$ the answer turned out to be positive.

In this paper we address the same question but for the case of non-integrable structures. In particular, we give sufficient conditions for such an approximation result to be valid locally in $(\mathbb{R}^{2n},J)$ (Theorem \ref{theoap}). We stress that, in contrast with the integrable case, a certain uniform bound is imposed on the $L^p$-norm of the differential $d\varphi$. The proof is based on the implicit function theorem for the linearization of the $\overline{\partial}_J$ operator and a careful study of the existence of a bounded right inverse (Theorem \ref{inverse}).  

Finally, motivated by \cite{mcdu-sa1,mcdu-sa2} we present in the last section an application of the above result. More precisely, we glue together two $J$-holomorphic halves of an unit disc in order to obtain one holomorphic object.  

\section{Preliminaries}
Throughout the paper we denote by $\Delta$  the open unit disc in $\C$. 
\subsection{Almost complex manifolds and pseudoholomorphic discs}
Let $M$ be a real smooth manifold $M$. An {\it almost complex structure} $J$ on $M$ is a $\left(1,1\right)$ tensor field
satisfying $J^{2}=-Id$. The pair $\left(M,J\right)$ is called an {\it almost complex manifold}. Let $J_{st}$ be the standard structure on $\R^{2n}$, that is, $(\mathbb{R}^{2n},J_{st})\cong \mathbb{C}^n$. A differentiable map $u\colon\left(M',J'\right) \longrightarrow \left(M,J\right)$ between two almost complex manifolds 
is {\it $\left(J',J\right)$-holomorphic} if it satisfies 
$$J\left(u\left(q\right)\right)\circ d_{q}u=d_{q}u\circ J'\left(q\right),$$ 
for every $q \in M'$, where $d_{q}u$ denotes the differential map of $u$ at $q$. When $M'$ is the unit disc $\Delta$, then such a map $u: \Delta \to (M,J)$ is called a {\it $J$-holomorphic disc}. Equivalently, $u$ is a $J$-holomorphic disc whenever the following non-linear operator vanishes  
$$\overline{\partial}_J u(v)=\frac{1}{2}\left(du(v)+J(u)du(J_{st} v)\right)=0.$$

 \subsection{The local equation}

Suppose that $J$ is a smooth almost complex structure defined in an open set $U\subseteq\R^{2n}$. Then it may be represented by a $\mathbb{R}$-linear operator $J(z)\colon \R^{2n} \to \mathbb{R}^{2n}$
satisfying $J(z)^2=-Id$. Further, the $J$-holomorphy equation for a $J$-holomorphic disc 
$u\colon \Delta \rightarrow U \subseteq \R^{2n}$ can be written as 
\begin{equation*}
\frac{\partial u}{\partial y}-J\left(u\right)\frac{\partial u}{\partial x}=0.
\end{equation*}
Moreover, we can rewrite it in its complex form
\begin{equation}\label{A}
\displaystyle u_{\bar{\zeta}}+A(u)\overline{u_{\zeta}}=0,
\end{equation}
where $\zeta=x+iy \in \C$ and 
$$A(z)(v)=(J_{st}+J(z))^{-1}(J(z)-J_{st})(\bar{v})$$
is a complex linear endomorphism for every $z\in U$ and $v \in \C^n$. Hence 
$A$ can be considered as a $n\times n$ complex matrix of the same regularity as $J(z)$ acting on $v\in\C^n$. We call $A$ the \emph{complex matrix of} $J$.  

Note that the above complex form (\ref{A}) is valid only when $J(z)+J_{st}$ is invertible. This, in particular, can be achieved locally by a change of coordinates in a neighborhood of any given point \cite[Lemma 1]{ga-su} or in a neighborhood of $u(\overline{\Delta})$ where $u\colon \overline{\Delta}\to \mathbb{R}^{2n}$ is an embedded $J$-holomorphic disc (see the Appendix in \cite{IR}); or globally, when $J$ is tamed by the standard symplectic form $\omega_{st}$ \cite{au-la} (see also \cite[Proposition 2.8]{ST2}). We denote by $\mathcal{J}$ the set of all smooth structures on $\mathbb{R}^{2n}$ satisfying such a condition and remark that it is in a one-to-one correspondence with the set of complex matrices $A$ satisfying the condition $\det(I-A\bar{A})\neq 0$ (see \cite{ST1}).

\subsection{Sobolev spaces and the Cauchy-Green operator}

Let $p>2$ and $k\in\N$. Let $\Omega\subset \mathbb{C}$ be bounded. We denote by $L^p(\Omega)$ the classical Lebesgue space and by $W^{k,p}(\Omega)$ the Sobolev space of maps $u\colon \Omega \to \mathbb{C}^n$ whose derivatives  up to order $k$ are in $L^p(\Omega)$. We sometimes abbreviate this to $L^p$ and $W^{1,p}$ if $\Omega$ is clear from the context.  The space $W^{k,p}(\Omega)$ is endowed with the usual norm 
$$\|u\|_{W^{k,p}(\Omega)}=\sum_{j=0}^k\|D^ju\|_{L^p(\Omega)}.$$
For $0<\alpha<1$, we denote by $\mathcal{C}^{k,\alpha}(\overline{\Omega})$ the H\"older space equipped with the norm
$$\|u\|_{\mathcal{C}^{k,\alpha}(\overline{\Omega})}=\sum_{j=0}^{k}\|D^ju\|_{L^\infty(\Omega)}+
\underset{\zeta\neq \eta}{\mathrm{sup}}\frac{\|D^ku(\zeta)-D^ku(\eta)\|}{|\zeta-\eta|^\alpha},$$
where $\|D^ju\|_{L^\infty(\Omega)}=\sup_{\zeta\in \Omega}\|D^ju(\zeta)\|$. Since $p>2$, the classical Sobolev embedding theorem 
ensures the existence of a positive constant $c>0$ such that for $u\in W^{1,p}(\Omega)$ we have 
$$\|u\|_{L^\infty(\Omega)}\leq \|u\|_{\mathcal{C}^{0,1-2/p}(\Omega)}\leq c \|u\|_{W^{1,p}(\Omega)}.$$
In particular, maps $u\in W^{1,p}(\Omega)$ have a bounded image $u(\Omega)\subset \C^n$.

In order to study Equation (\ref{A}) the main analytic tool is the {\it Cauchy-Green operator}
\begin{eqnarray*} T(u)(z)=\frac{1}{\pi}\int_{\Delta}\frac{u(\zeta)}{z-\zeta}\, \du x \, \du y, \end{eqnarray*}
defined for $u\colon\Delta \to \C^n$. We will need the following properties \cite{VEKUA}. 
 \begin{proposition}\label{propT} Let $p>2$.
 \begin{enumerate}[i.]
  \item The operator $T\colon L^{p}(\Delta) \to W^{1,p}(\Delta)$ is bounded and  $T\colon L^{p}(\Delta) 
  \to L^{\infty}(\Delta)$ is compact.
 \item  The Cauchy-Green operator $T$ solves the usual $\bar{\partial}$-equation, that is, for $f \in L^{p}(\Delta)$ we have 
  $(T(f))_{\bar{\zeta}}=f$ where the derivative is in the sense of Sobolev.
  \end{enumerate}
 \end{proposition}
 
\section{The bounded right inverse} 
Let us now turn for a moment to the approximation problem raised in the introduction. It is classical that it can be solved using the Cauchy-Green operator in the standard case $(\mathbb{R}^{2n},J_{st})\cong\mathbb{C}^n.$ Indeed, let $\varphi\colon\Delta\to\mathbb{C}^n$ be such that $\left\|\varphi_{\bar{\zeta}}\right\|_{L^p(\Delta)}<\delta$. Then its holomorphic approximation is given by 
$$u=\varphi-T(\varphi_{\bar{\zeta}}).$$
Note that by $i.$ Proposition \ref{propT} we have $\left\|u-\varphi\right\|_{W^{1,p}(\Delta)}<c\delta$, where the constant $c>0$ depends only on $p>2$. Thus, due to the Sobolev embedding theorem $u$ is also $\mathcal{C}^{0,1-2/p}(\overline{\Delta})$-close to the non-holomorphic map $\varphi$.
 
Essentially, the Cauchy-Green operator is a bounded right inverse of the usual $\bar{\partial}$ operator. Hence, one can find an appropriate small correction and add it to the initial disc. We will mimic this idea but for the linearization of Equation (\ref{A}). 

\begin{theorem}\label{inverse}
Let $J\in \mathcal{J}$ and let $A$ be its complex matrix. Let $\mathcal{F}:W^{1,p}(\Delta) \to L^p(\Delta)$ be the non-linear operator given by   
\begin{equation}\label{operator F}\mathcal{F}(u)=\displaystyle u_{\bar{\zeta}}+A(u)\overline{u_{\zeta}}.\end{equation}
Then for every $\varphi\in W^{1,p}(\Delta)$ the Fr\'echet derivative $d_\varphi\mathcal{F}$ admits a bounded right inverse $Q_\varphi$.
\end{theorem}
\begin{proof}
Let us start with a special case and assume that $\varphi\colon \Delta \to \C^n$ satisfies $A(\varphi)=0$, that is, the $J=J_{st}$ along the image $\varphi(\Delta)$. Then, the linearization of $\mathcal{F}$ at $\varphi$ is given by 
\begin{equation}\label{eqlin}
d_\varphi\mathcal{F}(h)=h_{\bar{\zeta}}+B_1^\varphi h + B_2^\varphi\bar{h},
\end{equation}
where $h\colon\Delta \to \C^n$ and $B_1^\varphi$ and $B_2^\varphi$ are two $n\times n$ matrix functions defined on 
$\Delta$ arising from the derivatives of $A$ and $\varphi$. Note that it is important that $\varphi \in W^{1,p}(\Delta)$ and that therefore $\varphi(\Delta)$ is relatively compact. Thus, $B_1^\varphi$ and $B_2^\varphi$ have rows and columns of class $L^p(\Delta)$. 

For $n=1$ a rather complete theory of such systems was given by I.N. Vekua \cite{VEKUA}. For $n\geq2$ the reader is referred to \cite{PASCALI,bo}. Using the Cauchy-Green operator we obtain an integral version of the above operator 
\begin{equation*}
\Phi_\varphi(h)=h+T(B_1^\varphi h + B_2^\varphi\bar{h}).
\end{equation*}
By $i.$ Proposition \ref{propT} the map $h\mapsto T(B_1^\varphi h + B_2^\varphi\bar{h})$ is compact. Hence $\Phi_\varphi$ is a Fredholm operator mapping the space $W^{1,p}(\Delta)$ to itself. Moreover, its Fredholm index equals to zero. Thus it is onto if and only of $\dim \textrm{Ker}(\Phi_\varphi)=0$. This, in particular is always true for $n=1$, but not for for $n\geq 2$. Still, a small linear and holomorphic part $L$ can be added in order to obtain invertibility. That is, there exists a linear operator $L\colon W^{1,p}(\Delta)\to W^{1,p}(\Delta)$  such that $\tilde{\Phi}_\varphi=\Phi_\varphi+L$ is invertible, $L(h)_{\bar{\zeta}}=0$ and $(\tilde{\Phi}_\varphi(h))_{\bar{\zeta}}=0$ if and only if $d_u\mathcal{F}(h)=0$ (see \cite[Theorem 3.1.]{ST1} for details). Moreover, we have
$$\Phi_\varphi\left(\tilde{\Phi}_\varphi^{-1} \circ T (h)\right)+L(\tilde{\Phi}_\varphi^{-1} \circ T (h))=T (h).$$ 
Since the term $L(\tilde{\Phi}_\varphi^{-1} \circ T (h))$ is holomorphic, differentiating the previous in $\overline{\zeta}$ 
leads to 
$$d_\varphi\mathcal{F}(\tilde{\Phi}_\varphi^{-1} \circ T (h))=h.$$ 
Thus the operator $Q_\varphi=\tilde{\Phi}_\varphi^{-1} \circ T: L^p(\Delta) \to W^{1,p}(\Delta)$ is a 
bounded right inverse of the operator $d_\varphi\mathcal{F}$.  

Assume now that $A(\varphi)$ does not vanish identically. Following \cite{ST1} p.9, we introduce a substitution by real linear transformation $N\colon W^{1,p}(\Delta)\to W^{1,p}(\Delta)$:
$$N(u)=u+A(\varphi)\overline{u}.$$
Note that it is well defined and continuous since $\varphi(\Delta)$ is relatively compact and $\varphi\in W^{1,p}(\Delta)$. Moreover, it is  one-to-one and one can check that 
$$N^{-1}(v)=B(v-A(\varphi)\bar v).$$
where $$B=\left(I-A(\varphi)\overline{A(\varphi)}\right)^{-1}.$$ 
We write $\mathcal{F}=F\circ N$, with  
$$F(v)=K(v)v_{\bar\zeta}+K_0(v)\overline{v_\zeta}+K_1(v) v+K_2(v)\bar v,$$
where
$$
\left\{
\begin{array}{lll} 
K(v)&=&B-A(N^{-1}(v))\overline{BA(\varphi)}\\
\\
K_0(v)&=&-BA(\varphi)+A(N^{-1}(v))\overline{B}\\ 
\\
K_1(v)&=&B_{\overline{\zeta}}-A(N^{-1}(v))\overline{B_{\zeta}} \ \overline{A(\varphi)}-
A(N^{-1}(v))\overline{B} \ \overline{A(\varphi)_{\zeta}}\\
\\
K_2(v)&=&-B_{\overline{\zeta}}A(\varphi)-BA(\varphi)_{\overline{\zeta}}+A(N^{-1}(v))\overline{B_{\zeta}}.
\end{array}
\right.$$
The key property of the transformation $N$ is that if we set
$$\psi=N(\varphi)=\varphi+A(\varphi)\overline{\varphi},$$
then we have $K(\psi)=I$ and $K_0(\psi)=0$. Hence the differential at $\psi$ of the new map $F\colon W^{1,p}(\Delta) \to L^p(\Delta)$ is again of the form
$$d_\psi F(h)=h_{\bar\zeta}+B^\psi_1 h+B^\psi_2 \bar h$$
with the matrix functions $B^\psi_1$ and $B^\psi_2$ whose rows and columns are of the class $L^{p}(\Delta)$. The existence of a bounded right inverse $Q_{N(\varphi)}$ for such an operator was already proved above. Hence $Q_\varphi=N^{-1}\circ Q_{N(\varphi)}$
is a bounded right inverse of $d_{\varphi}\mathcal{F}$. 
\end{proof}

For $J$-holomorphic discs the above theorem generalizes importantly to the case of almost complex manifolds. Indeed, note that under the assumptions of Theorem \ref{inverse} the differential of $\overline{\partial}_J$ at an embedded $J$-holomorphic disc admits a bounded right inverse in the Euclidean space $(\R^{2n},J)$. More generally, let $M$ be an almost complex manifold $M$ and $\varphi$ an embedded $J$-holomorphic disc of class $W^{1,p}$. By \cite{IR} one can choose coordinates around $\varphi(\Delta)$ such that $J(\varphi)=J_{st}$. Hence, the operator $D_\varphi$ defined in 3.1 \cite{mcdu-sa2} p. 38 admits a bounded right inverse. The proof is essentially the same as the one above for $A(\varphi)=0$. However, there are two main differences between the operator $D_\varphi$ in our (boundary free) case and the one defined for compact curves in \cite{mcdu-sa2}. Firstly, in our case $D_\varphi$ is always onto (even for non-holomorphic $W^{1,p}$-class maps represented locally in a chart with $J\in\mathcal{J}$). Secondly, for discs, the operator $D_\varphi$ is never Fredholm unless we impose additional totally real boundary conditions and reduce its kernel to a finite dimension. Nevertheless, as shown below we can always associate a Fredholm integral form to the non-linear operator.

Let $\mathcal{F}$ be the operator defined in (\ref{operator F}). Using the Cauchy-Green operator we associate to it the operator $
\mathcal{G}\colon W^{1,p}(\Delta)\to W^{1,p}(\Delta)$ given by
\begin{equation}\label{integral form}
\mathcal{G}(u)=u+T(A(u)\overline{u\zeta}).
\end{equation}
Since $J$ is smooth, $\mathcal{G}$ depends smoothly on $u\in W^{1,p}(\Delta)$. Moreover, it follows from the proof above that given 
$\varphi\in W^{1,p}(\Delta)$ the derivative $d_\varphi\mathcal{G}$ is an index zero Fredholm map but with possibly non-trivial kernel (see the example below). Note that the last does not object the existence of the right inverse $Q_\varphi$, but in general its continuous 
dependence on $\varphi\in W^{1,p}(\Delta)$ may be questionable. It will be important for us to omit such cases in Section $4$ where 
a bound on the norm of  $Q_\varphi$ is needed. Hence, we introduce some additional terminology; one should note that what follows 
is not needed for the main result presented in the next section. 

We say that the pair $(\varphi,J)\in W^{1,p}(\Delta)\times\mathcal{J}$ is \textit{regular} if the map $d_\varphi\mathcal{G}$ is onto. We 
say that $J$ is \textit{regular on} $\mathcal{W}\subset W^{1,p}(\Delta)$ if $(\varphi,J)$ is regular for every $\varphi\in\mathcal{W}$. 
We denote by $\mathcal{J}_{reg}(\mathcal{W})$ the set of all such structures. From the above discussion one can deduce the 
following statement.  
\begin{coro}\label{continuous}
Let $\mathcal{W}\subset W^{1,p}(\Delta)$ and $J\in\mathcal{J}_{reg}(\mathcal{W})$. Then the derivative $d_\varphi \mathcal{F}$ of (\ref{operator F}) admits a bounded right inverse whose norm $\left\|Q_\varphi\right\|$ depends continuously on $\varphi\in\mathcal{W}$.   
\end{coro}     

Let us further justify the notion of regularity by giving an explicit example motivated by \cite{HABETHA,BUCHANAN}, in which we provide a non-regular pair. 
\vskip 0.1 cm
\noindent\textbf{Example:}  Let $J\in\mathcal{J}$ be an almost complex structure on $\mathbb{R}^{6}$ corresponding to the complex matrix 
$$ A(z_1,z_2,z_3)=\left[\begin{array}{ccc}0 & 0 & 0\\ \displaystyle \frac{6z_1^2z_3}{3-z_1^2\bar{z}_1^2} & 0 & 0\\ -z_2& 0 & 0\end{array}\right]$$
when restricted the unit ball of $\mathbb{R}^6$. Let $\varphi\colon \Delta\to \mathbb{C}^3$ be a $J$-holomorphic disc given by $\varphi(\zeta)=(\zeta,0,0)$.
Then the derivative $d_\varphi\mathcal{G}$ of the operator (\ref{integral form}) at $\varphi$ is equal to
$$d_\varphi\mathcal{G}(h)=h+T\left(\left[\begin{array}{ccc}0 & 0 & 0\\ 0 & 0 & \displaystyle \frac{6\zeta^2}{3-\zeta^2\bar{\zeta}^2}\\ 0&-1&0\end{array}\right]h\right).$$
Hence $h=(h_1,h_2,h_3)\in\ker d_\varphi\mathcal{G}$ implies $h_1=0$ and for $h'=(h_2,h_3)$ we obtain a system  
\[
	h'_{\bar{\zeta}}+ 
	\begin{pmatrix} 
		0 &6\zeta^2/(3-\zeta^2\bar{\zeta}^2) \\ 
		-1&0
	\end{pmatrix}h'=0.
\]
Let us differentiate the second row with respect to $\bar{\zeta}$:
$$(h_3)_{\bar{\zeta}\bar{\zeta}}=(h_2)_{\bar{\zeta}}= \frac{6\zeta^2}{\zeta^2\bar{\zeta}^2-3}h_3.$$
This second order equation admits two particular solutions. The first one is $\psi_1(\zeta)=\bar{\zeta}-\frac{1}{3}\zeta^2\bar{\zeta}^3$ and the second one equals to $\psi_2(\zeta)=\sum_{k=0}^\infty{b_k}\zeta^{2k}\bar{\zeta}^{2k},$
where the coefficients $b_k$ are defined by 
$$b_0=1,\;\;\;\; b_k=\prod_{j=1}^k\frac{(j-1)(2j-3)-3}{3(2j-1)j}.$$
(Since $0<b_k<\frac{1}{3^k}$ for $k\geq 2$, this series converges for $\left|\zeta\right|\leq 3,$ as do its derivatives). Thus the general solution of the original system is given by the expression
\begin{eqnarray*}
\left[\begin{array}{c}h_2(\zeta)\\h_3(\zeta)\end{array}\right]=\left[\begin{array}{c}\mu_{\bar{\zeta}}(\zeta)\phi_1(\zeta)+\psi_{\bar{\zeta}}(\zeta)
\phi_2(\zeta)\\ \mu(\zeta)\phi_1(\zeta)+\psi(\zeta)\phi_2(\zeta) \end{array}\right], 
\end{eqnarray*}
where $\phi_1,\phi_2$ are holomorphic functions on $\Delta$. In particular, for $\zeta\in\partial\Delta$ we have
\begin{eqnarray}\label{splosna resitev rob}
\left[\begin{array}{c}h_1(\zeta)\\h_2(\zeta)\end{array}\right]=\left[\begin{array}{c}\lambda_1\zeta\phi_2(\zeta) \\ \frac{2}{3}\bar{\zeta}\phi_1(\zeta)+\lambda_2\phi_2(\zeta) \end{array}\right],
\end{eqnarray}
where $\lambda_1=2\sum_{k=0}^{\infty}kb_k$ and $\lambda_2=\sum_{k=0}^{\infty}b_k.$ Finally, note that by the generalized Cauchy integral formula $h\in\ker d_\varphi\mathcal{G}$ if and only if for every $z\in\Delta$
$$C(h_j)(z)=\oint_{\partial{\Delta}}\frac{h_j(\zeta)\mathrm{d}\zeta}{\zeta-z}=0, \;\;  j\in\left\{1,2,3\right\}.$$
For (\ref{splosna resitev rob}) this is fulfilled if $\phi_2\equiv 0$ and $\phi_1$ is a complex constant. Hence $\dim_{\R}\ker d_\varphi\mathcal{G}=2.$

\section{Approximation of non $J$-holomorphic maps}
In this section we prove our main result. We start by recalling the following Newton-Picard iteration type theorem to find zeros of functionals in Banach 
spaces.
\begin{theorem}[Proposition A.3.4 from \cite{mcdu-sa2}]
\label{theoimpl}
Let $X$ and $Y$ be two Banach spaces and consider a  
map $\mathcal{F}: U\subset X \to Y$ of class $\mathcal{C}^1$ defined on an open set $U\subset X$. 
Let $x_0\in U$. Assume that the differential $d_{x_0}\mathcal{F}$ admits a bounded right inverse, denoted by $Q_{x_0}$. Fix $c_0>0$ such that $\left\|Q\right\|\leq c_0$ and $\eta>0$ such that if $\left\|x-x_0\right\|<\eta$ then $x\in U$ and $$\left\|d_x\mathcal{F}-d_{x_0}\mathcal{F}\right\|\leq \frac{1}{2c_0}.$$
Then if $\displaystyle \left\|\mathcal{F}(x_0)\right\|<\frac{\eta}{4c_0}$ there exists $x\in U$ such that $\mathcal{F}(x)=0$ and 
$$\left\|x-x_0\right\|\leq2c_0\left\|\mathcal{F}(x_0)\right\|.$$
\end{theorem}
\noindent We now state our main theorem. 

\begin{theorem}\label{theoap}
Let $p>2$. Let $J\in \mathcal{J}$ be an almost complex structure on $\mathbb{R}^{2n}$ and let $A$ be its complex matrix. Let $\mathcal{F}:W^{1,p}(\Delta) \to L^p(\Delta)$ be the operator given by 
$$\mathcal{F}(u)=\displaystyle u_{\bar{\zeta}}+A(u)\overline{u_{\zeta}}.$$
For every $c_0>0$, there exists $\delta>0$ such that for any $\varphi \in W^{1,p}(\Delta)$ satisfying 
$$\|d\varphi\|_{L^p(\Delta)}\leq c_0,\;\; \|Q_\varphi\|\leq c_0,\;\;  \left\|\mathcal{F}(\varphi)\right\|_{L^{p}(\Delta)}<\delta,$$
where $Q_\varphi$ is a bounded right inverse  of $d_{\varphi}\mathcal{F}$,   
 there exists a $J$-holomorphic disc $u\in W^{1,p}(\Delta)$  such that 
$$\left\|u-\varphi\right\|_{W^{1,p}(\Delta)}\leq 2c_0\left\|\mathcal{F}(\varphi)\right\|_{L^{p}(\Delta)}.$$ 

\end{theorem}

\begin{proof}
We apply  Theorem \ref{theoimpl} for $X= W^{1,p}(\Delta)$, $Y=L^p(\Delta)$ and  
$x_0=\varphi \in W^{1,p}(\Delta)$. Fix a positive constant $c_0>0$ and assume that $\varphi \in W^{1,p}(\Delta)$ satisfies   
$$\|d\varphi\|_{L^p(\Delta)}\leq c_0 \mbox{ } \mbox{ and } \|Q_\varphi\|\leq c_0.$$
We claim that there exists a positive constant $c>0$  such that 
$$\left\|d_{\tilde{\varphi}}\mathcal{F}-d_{\varphi}\mathcal{F}\right\|\leq c\|\tilde{\varphi}-\varphi\|_{W^{1,p}(\Delta)}$$ for any $\tilde{\varphi} \in W^{1,p}(\Delta)$ in a $W^{1,p}$-neighborhood of $\varphi$, say $\displaystyle \|\tilde{\varphi}-\varphi\|_{W^{1,p}(\Delta)}< 1$. Let $h \in W^{1,p}(\Delta)$, we need to prove that 
\begin{equation}\label{eqlip}
\displaystyle \left\|d_{\tilde{\varphi}}\mathcal{F}(h)-d_{\varphi}\mathcal{F}(h)\right\|_{L^p(\Delta)}\leq c\|\tilde{\varphi}-\varphi\|_{W^{1,p}(\Delta)}\|h\|_{W^{1,p}(\Delta)}.
\end{equation}
 Note that we have already use the general form (\ref{eqlin}) to express $d_{\varphi}\mathcal{F}(h)$. However in order to show (\ref{eqlip}), 
 we need to be more precise. We have   
$$ 
d_{\varphi}\mathcal{F}(h) =  h_{\bar \zeta}+A(\varphi)\overline{h_{\zeta}}+d_\varphi A ( h)  \ \overline{\varphi_\zeta},
$$
where $\displaystyle d_\varphi A  (h)=\sum_{j=1}^n\frac{\partial A}{\partial z_j}(\varphi)h_j+\frac{\partial A}{\partial \bar{z}_j}(\varphi)\bar{h}_j$.
We write $d_{\tilde{\varphi}}\mathcal{F}(h)-d_{\varphi}\mathcal{F}(h)=I+II$ with
$$
\left\{
\begin{array}{lll} 
I&= &\left(A(\tilde{\varphi})-A(\varphi)\right)\overline{h_{\zeta}}+(d_{\tilde{\varphi}} A-d_\varphi A)(h)\overline{\tilde{\varphi}_\zeta}\\
\\
II &= &d_{\varphi} A  (h)\left(\overline{\tilde{\varphi}_\zeta-\varphi_\zeta}\right).
\end{array}
\right.$$                                                                                                                                           
Let us denote by $\left\|A\right\|_{\infty}$ the maximum absolute value taken over the coefficients of the matrix map $\zeta\mapsto A(\zeta)$ and $\zeta\in\overline{\Delta}.$ Since $\|\tilde{\varphi}-\varphi\|_{W^{1,p}(\Delta)}< 1$ and since $\varphi(\Delta)$ is relatively compact, it follows that there exists a positive constant $c_1>0$, depending on $A$ and $\varphi$, such that
$$  \|A(\tilde{\varphi})-A(\varphi)\|_{\infty} \leq c_1\|\tilde{\varphi}-\varphi\|_{W^{1,p}(\Delta)},$$ 
$$   \left\|(d_{\tilde{\varphi}} A-d_\varphi A)(h)\right\|_{\infty}\leq c_1\|\tilde{\varphi}-\varphi\|_{W^{1,p}(\Delta)}\left\|h\right\|_{L^{\infty}(\Delta)}.$$ 
Moreover since  $\|d\varphi\|_{L^p(\Delta)}\leq c_0$, we have
  $$ \|d\tilde{\varphi}\|_{L^p(\Delta)}\leq \|d\tilde{\varphi}-d\varphi\|_{L^p(\Delta)} +\|d\varphi\|_{L^p(\Delta)}< 1+c_0.$$
This leads to 
$$ \left\|I\right\|_{L^p(\Delta)}\leq c_1(1+c_2+c_0c_2)\|\tilde{\varphi}-\varphi\|_{W^{1,p}(\Delta)}\left\|h\right\|_{W^{1,p}(\Delta)},$$
where the constant $c_2>0$ arises from the Sobolev embedding theorem.
Moreover, there exist a constant $c_3>0$ such that
$$\|d_{\varphi} A(h)\|_{\infty}\leq c_3 \left\|h\right\|_{L^{\infty}(\Delta)}$$
and finally, we have
$$\left\|II\right\|_{L^p(\Delta)}\leq c_2c_3 \|\tilde{\varphi}-\varphi\|_{W^{1,p}(\Delta)}\left\|h\right\|_{W^{1,p}(\Delta)}.$$
Thus the inequality (\ref{eqlip}) follows. Finally, set $\eta=\min\{1,1/(2cc_0)\}$ and $\displaystyle \delta=\frac{\eta}{4c_0}$. The desired result follows then from Theorem \ref{theoimpl}.
\end{proof}

\section{Gluing together two halves of a disc} 

In this simple example we demonstrate how our Theorem \ref{theoap} can be used for gluing constructions. Here, by gluing we mean finding a $J$-holomorphic map whose image is close to the disjoint union of two given $J$-holomorphic objects (see for instance \cite{mcdu-sa1,mcdu-sa2} for the case of $J$-holomorphic spheres).    

Let us fix $\tau>0$ and define two overlapping halves of the unit disc 
$$\Delta_1=\Delta\cap\left\{\textrm{Re}(\zeta)>-\tau\right\}, \;\; \Delta_2=\Delta\cap\left\{\textrm{Re}(\zeta)<\tau\right\}.$$
Let $\Omega\in\left\{\Delta,\Delta_1,\Delta_2\right\}$. We denote by $\mathcal{B}_{2,p}^M(\Omega)$ the ball of radius $M>0$ in  $W^{2,p}(\Omega)$, $p>2$. Note that  such a ball can be compactly embedded into the space $W^{1,p}(\Omega)$. Our goal is to glue two halves of a disc $u_j\in \mathcal{B}_{2,p}^{M/2}(\Delta_j)$, $j=1,2$, whose difference is sufficiently small on the intersection $\Delta_1\cap\Delta_2$, into one holomorphic object. However, similarly to \cite{mcdu-sa1,mcdu-sa2} not every almost complex structure is suitable for such a construction. In particular, we restrict to the structures that are regular in a neighborhood of the closed ball $\mathcal{B}^M_{2,p}(\Delta)$ in order to obtain an uniform bound for the norm of the right inverse by Corollary \ref{continuous}.   
\begin{proposition}
Let $\epsilon>0$, $p>2$ and $M>0$. Let $\mathcal{W}\subset W^{1,p}(\Delta)$ be an open set containing the closure of the ball $\mathcal{B}_{2,p}^M(\Delta)$ and let $J\in \mathcal{J}_{reg}(\mathcal{W})$. There exists a constant $\delta_0=\delta_0(\tau,\epsilon,M,J)>0$ such that for every pair of $J$-holomorphic maps $u_j\in \mathcal{B}^{M/2}_{2,p}(\Delta_j)$, $j=1,2$, satisfying $\left\|u_1-u_2\right\|_{W^{1,p}(\Delta_1\cap\Delta_2)}<\delta_0$ there exists a $J$-holomorphic map $u\colon\Delta\to\mathbb{R}^{2n}$ such that $\left\|u-u_j\right\|_{W^{1,p}(\Delta_j)}<\epsilon$. 
\end{proposition}
\begin{proof}
Let us fix a smooth cut-off function $\chi\colon\Delta\to[0,1]$ such that $\chi=1$ on $\Delta_1\setminus \Delta_2$ and $\chi=0$ on $\Delta_2\setminus \Delta_1$. Let $c_1=c_1(\tau)>0$ be such that $\left\|\chi\right\|_{\mathcal{C}^{2}(\Delta)}<c_1$. Consider a 
pair of $J$-holomorphic maps $u_j\in \mathcal{B}^{M/2}_{2,p}(\Delta_j)$, $j=1,2$. We define a pre-gluing map $\varphi\colon\Delta\to \mathbb{R}^{2n}$ by
$$\varphi(\zeta)=\chi u_1 + (1-\chi) u_2.$$
The idea is to seek its holomorphic approximation.

Since $W^{2,p}(\Delta)\subset \mathcal{C}^{1,1-2/p}(\overline{\Delta})\subset W^{1,p}(\Delta)$ by Theorem \ref{inverse} the derivative $d_\varphi\mathcal{F}$ admits a bounded right inverse $Q_\varphi$. Moreover, for $\left\|u_1-u_2\right\|_{W^{1,p}(\Delta_1\cap\Delta_2)}<\delta$ we have
$$\left\|\varphi\right\|_{W^{2,p}(\Delta)}\leq \sum_{j=1,2}\left\|u_j\right\|_{W^{2,p}(\Delta_j)}+ 3\left\|\chi\right\|_{\mathcal{C}^{2}(\Delta)}\left\|u_1-u_2\right\|_{W^{1,p}(\Delta_1\cap\Delta_2)}< M+3c_1\delta.$$
Let $\delta_1=\delta_1(\tau,\mathcal{W})>0$ be such that $\mathcal{B}_{2,p}^{M+\delta_1}(\Delta)\subset \mathcal{W}$ is compactly embedded. Then, for $\delta<\frac{\delta_1}{3c_1}$, every pre-gluing map is contained in a compact subset of the space $\mathcal{W}$. Thus by  Corollary \ref{continuous} we have a constant  $c_0=c_0(M,J)>0$ such that
$$\left\|d\varphi\right\|_{L^{p}(\Delta)}<c_0\;\;\;\textrm{  and  }\;\;\; \left\|Q_\varphi\right\|<c_0.$$  

On the other hand, the map $\mathcal{F}(\varphi)$ vanishes everywhere except on the intersection $\Delta_1\cap \Delta_2$. Moreover, on that particular set we have $\mathcal{F}(\varphi)=I+II$ with
$$
\left\{
\begin{array}{lll} 
I&=&\chi_{\bar{\zeta}}(u_1-u_2)+A(\varphi)\overline{\chi_\zeta(u_1-u_2)}\\
\\
II&=&\chi((u_1)_{\bar{\zeta}}+A(\varphi)\overline{(u_1)_\zeta})+(1-\chi)((u_2)_{\bar{\zeta}}+A(\varphi)\overline{(u_2)_\zeta}).\\
\end{array}
\right.$$
Hence,
$$\left\|\mathcal{F}(\varphi)\right\|_{L^p(\Delta)}=\left\|\mathcal{F}(\varphi)\right\|_{L^p(\Delta_1\cap \Delta_2)}<\left\|I\right\|_{L^p(\Delta_1\cap U_2)}+\left\|II\right\|_{L^p(\Delta_1\cap \Delta_2)}.$$
Furthermore, we have
$$\left\|I\right\|_{L^p}\leq\left(\left\|\chi_{\bar{\zeta}}\right\|_{L^\infty}+\left\|A(\varphi)\right\|_{L^\infty}\left\|\chi_{\zeta}\right\|_{L^\infty}\right)\left\|u_1-u_2\right\|_{L^p}<c_2(\tau,M,J)\delta,$$
and
$$\left\|II\right\|_{L^p}\leq \left\|A(\varphi)-A(u_1)\right\|_{L^\infty}\left\|(u_1)_\zeta\right\|_{L^p}+\left\|A(\varphi)-A(u_2)\right\|_{L^\infty}\left\|(u_2)_\zeta\right\|_{L^p}<c_3(\tau,M,J)\delta.$$
By Theorem \ref{theoap} there exists $\delta_2>0$ such that if
$$\left\|\mathcal{F}(\varphi)\right\|_{L^p(\Delta)}<\delta_2$$
 we can find a $J$-holomorphic disc $u\colon\Delta\to\mathbb{R}^{2n}$ with
$$\left\|u-\varphi\right\|_{W^{1,p}(\Delta)}\leq 2c_0\left\|\mathcal{F}(\varphi)\right\|_{L^p(\Delta)}.$$
Hence for  $\displaystyle \delta<\delta_0=\min\left\{\frac{\delta_1}{3c_1},\frac{\delta_2}{c_2+c_3},\frac{\epsilon}{4c_0(c_2+c_3)},\frac{\epsilon}{2(1+c_1)}\right\}$ we have a $J$-holomorphic disc $u \in W^{1,p}(\Delta)$ that is $\epsilon/2$-close to $\varphi$ in $W^{1,p}(\Delta)$. 
Moreover, we have
$$\left\|u-u_j\right\|_{W^{1,p}(\Delta_j)}\leq \left\|u-\varphi\right\|_{W^{1,p}(\Delta_j)}+\left\|\varphi-u_j\right\|_{W^{1,p}(\Delta_j)}<\frac{\epsilon}{2}+(1+c_1)\delta<\epsilon.$$ 
Hence the statement is proved. 
\end{proof}

\noindent\textbf{Final remark:} This paper was motivated by \cite{Rosay2} where J.-P. Rosay emphasized that his hope was to provide a proof of Poletsky theorem \cite{Poletsky1991} that could adapt to the more general setting of almost complex manifolds. Hence, let us give here a brief explanation on what seems to still be the missing step towards realizing his program. 

The above gluing result can obviously be generalized to the case of finitely many subsets with empty triple intersections. Moreover, we believe that a bounded right inverse $Q_\varphi$ for the operator $d_\varphi \bar{\partial}_J$ can be found even when the pre-gluing map $\varphi$ maps to a manifold (and hence one can avoid solving a non-linear Cousin problem used in the original paper). Nevertheless, what remains unclear is how to approximately attach a $J$-holomorphic discs to real tori (the Riemann-Hilbert problem), since J.-P. Rosay uses a family of discs $\left\{\varphi_N\right\}_{N\in\mathbb{N}}$ with their $\bar{\partial}_J$-derivatives tending to zero when $N\to\infty$ but with no bound on the $L^p$-norm of $d\varphi_N$. 

\vspace{0.1cm}

\thanks{Research of the first named author was supported by a long-term faculty development grant from the American University of Beirut thanks to which he visited the University of Ljubljana and the University of Vienna in the Summer 2016. 

Research of the second named author was supported in part by the
research program P1-0291 and the grant J1-7256 from ARRS, Republic of Slovenia. A large part of the result was created during his  bilateral visit at the University of Vienna, Summer 2016 (BI-AT/16-17-026) and during his stay at the University of Oslo, Spring 2017.

Both authors thanks these institutions for their support and hospitality.}

{\small
\noindent Florian Bertrand\\
Department of Mathematics\\
American University of Beirut, Beirut, Lebanon\\
{\sl E-mail address}: fb31@aub.edu.lb\\

\noindent Uro\v{s} Kuzman \\ 
Faculty of Mathematics and Physics\\
University of Ljubljana, Jadranska 19, 1000 Ljubljana, Slovenia\\
{\sl E-mail address}: 	uros.kuzman@gmail.com\\
}

\end{document}